\documentclass[12pt,reqno]{amsart}

\usepackage{amssymb}
\usepackage{amsthm}
\usepackage{amsmath}
\usepackage{a4wide}
\usepackage{latexsym}
\usepackage{mathrsfs}
\usepackage{euscript}
\usepackage[mathx]{mathabx}
\usepackage[v2,tips]{xy}
\usepackage{float}
\usepackage[T1]{fontenc}
\usepackage[latin1]{inputenc}

\newtheorem{theorem}{Theorem}[section]
\newtheorem{lemma}[theorem]{Lemma}
\newtheorem{proposition}[theorem]{Proposition}
\newtheorem{corollary}[theorem]{Corollary}

\theoremstyle{definition}
\newtheorem{definition}[theorem]{Definition}

\newtheorem{remark}[theorem]{Remark}

\numberwithin{equation}{section}

\newcounter{smallromans}

\newenvironment{romanenumerate}
{\begin{list}{{\normalfont\textrm{(\roman{smallromans})}}}%
    {\usecounter{smallromans}\setlength{\itemindent}{0cm}%
      \setlength{\leftmargin}{5.5ex}\setlength{\labelwidth}{5.5ex}%
      \setlength{\topsep}{0.2ex}\setlength{\partopsep}{0ex}%
      \setlength{\itemsep}{0.2ex}}}%
  {\end{list}}

\newcounter{smallalphs}

  {\end{list}}

\renewcommand{\leq}{\ensuremath{\leqslant}}

\renewcommand{\geq}{\ensuremath{\geqslant}}

\newcommand{\N}{\mathbb{N}}

\newcommand{\R}{\mathbb{R}}
\newcommand{\C}{\mathbb{C}}

\renewcommand{\phi}{\ensuremath{\varphi}}
\renewcommand{\epsilon}{\ensuremath{\varepsilon}}

\newcommand{\spa}{\operatorname{span}}

\begin{document}
\title[The uniqueness-of-norm problem for Calkin algebras]{The uniqueness-of-norm problem for Calkin algebras} \subjclass[2010]%
{Primary
  47L10, % Algebras of operators on Banach spaces and other topological linear spaces
  46H10, % Ideals and subalgebras
    46H40; % Automatic continuity
 Secondary   46B25. %Classical Banach spaces in the general theory
% 46B03. %Isomorphic theory of Banach spaces
  } 
\author[R.~Skillicorn]{Richard Skillicorn} \address{Department of
  Mathematics and Statistics, Fylde College, Lancaster University,
  Lancaster, LA1 4YF, United Kingdom}
\email{r.skillicorn@lancaster.ac.uk}
\begin{abstract} 
We examine the question of whether the Calkin algebra of a Banach space must have a unique complete algebra norm. We present a survey of known results, and make the observation that a recent Banach space construction of Argyros and Motakis (preprint, 2015) provides the first negative answer. The parallel question for the weak Calkin algebra also has a negative answer; we demonstrate this using a Banach space of Read (\emph{J.~London Math.\ Soc.}~1989).
\end{abstract}

\maketitle

\section{History of the Problem}%
\label{section1}
\noindent 
Uniqueness-of-norm questions for Banach algebras have been studied for almost as long as Banach algebras themselves. Eidelheit was the first to publish on the topic; his 1940 paper \cite{Eid} showed that the Banach algebra of bounded operators on a Banach space has a unique complete norm, and prompted the natural question of which other Banach algebras share this property. Gel'fand \cite{Gel} quickly followed in 1941 with a similar result in the commutative case: his famous proof that commutative, semisimple Banach algebras have a unique complete norm. 
It was Rickart \cite{Ric} who sought to tie the two ideas together in the 1950's by focusing on the problem of whether a (not necessarily commutative) semisimple Banach algebra always has a unique complete norm; this was solved positively by Johnson in 1967 \cite{johnson1}, and his result remains the major achievement in the area. 
%In particular, Johnson's result shows that all $C^*$-algebras and all ideals of operators on Banach spaces have a unique complete norm. 
Beyond the semisimple setting (that is, once the Jacobson radical is non-zero), things are much less clear. Dales and Loy \cite{DL} developed theory to handle certain cases when the Jacobson radical is finite-dimensional, however, they noted that even when the radical is one-dimensional a Banach algebra may lack a unique complete norm.  Despite the considerable amount of work done on the uniqueness-of-norm problem, much remains unknown. In this note we address the question for two particular classes of Banach algebras: the Calkin algebra and weak Calkin algebra of a Banach space. Yood observed that in general a Calkin algebra need not be semisimple \cite{Yood}, and the same holds true for weak Calkin algebras, so the question has no immediate answer. As we shall see, there is a good reason for this.
It is time to be precise about some definitions. 

\begin{definition} A Banach algebra $(A,||\cdot||)$ has a \emph{unique complete norm} if any other complete algebra norm on $A$ is equivalent to $||\cdot||$.
\end{definition}

Recall that two norms $||\cdot||$ and $|||\cdot|||$ on a vector space $A$ are \emph{equivalent} if there exist constants $c,C>0$ satisfying $c||a||\leq |||a|||\leq C||a||$ for every $a\in A$. In the case where $||\cdot||$ and $|||\cdot|||$ are complete, it is enough that one of these inequalities holds by the Banach Isomorphism Theorem.

Let $X$ be a Banach space. The scalar field is denoted by $\mathbb{K}$ throughout, which is either $\R$ or $\C$. We write $\mathscr{B}(X)$ for the Banach algebra of bounded, linear operators on $X$, $\mathscr{K}(X)$ for the closed ideal of compact operators, and $\mathscr{W}(X)$ for the closed ideal of weakly compact operators.  

\begin{definition} Let $X$ be a Banach space. The \emph{Calkin algebra} of $X$ is the Banach algebra $\mathscr{B}(X)/\mathscr{K}(X)$, and the \emph{weak Calkin algebra} of $X$ is the Banach algebra $\mathscr{B}(X)/\mathscr{W}(X)$. \end{definition}

Let us be clear about the questions we are considering.
\begin{enumerate}
\item Given a Banach space, must its Calkin algebra have a unique complete norm?
\item Must its weak Calkin algebra have a unique complete norm?
\end{enumerate}

We refer to this as the \emph{uniqueness-of-norm problem} for Calkin algebras.
% as a special case of the general problem for Banach algebras (see \emph{e.g.} \cite[p.~154]{da}).
The questions have their roots in Calkin's study \cite{Cal} of the Banach algebra $\mathscr{B}(H)/\mathscr{K}(H)$ where $H$ is a separable Hilbert space. He showed that there are no proper, non-trivial closed ideals in $\mathscr{B}(H)/\mathscr{K}(H)$; this implies that $\mathscr{B}(H)/\mathscr{K}(H)$ has a unique complete norm since it is semisimple. Once Yood had observed that there are non-semisimple Calkin algebras, Kleinecke defined the ideal of \emph{inessential operators} on a Banach space $X$, denoted $\mathscr{E}(X)$, as a `measure of non-semisimplicity' \cite{Klein}. More specifically, if $\mathscr{E}(X)=\mathscr{K}(X)$ then the Calkin algebra of $X$ is semisimple. The questions have also been considered indirectly by Johnson \cite{johnson}, Tylli \cite{Ty} and Ware \cite{Ware}, in the context of wider work.
%The second question is a natural outworking of the first, and has been similarly examined by Tylli \cite{Ty3}.

In the remainder of this section we give an overview of some known results. The author is not aware of any other survey of the same material, and it seemed helpful to draw together the scattered literature.
We begin with a large class of Banach spaces for which the problem is quickly solved, as our first proposition, which is due to Johnson \cite{johnson}, shows.
We give a proof to demonstrate some of the ideas under consideration.

\begin{proposition} \label{isomorphictosquare} Let $X$ be a Banach space such that $X\cong X\oplus X$. Then $\mathscr{B}(X)/\mathscr{K}(X)$ and $\mathscr{B}(X)/\mathscr{W}(X)$ have a unique complete norm.
\end{proposition}
\begin{proof}  Let $I$ be a closed ideal of $\mathscr{B}(X)$ and consider a complete algebra norm $|||\cdot|||$ on $\mathscr{B}(X)/I$ different from the quotient norm $||\cdot||$. By \cite[Theorem 3.3; see remarks after Theorem 3.5]{johnson} the identity map $\iota: (\mathscr{B}(X)/I,||\cdot||)\to (\mathscr{B}(X)/I,|||\cdot|||)$ is continuous. So there exists $C>0$ such that $|||b+I||| \leq C||b+I||$ for every $b\in \mathscr{B}(X)$, and hence
%Next, the Banach Isomorphism Theorem yields $c>0$ such that $||b+I|| \leq c|||b+I|||$ for every $b\in \mathscr{B}(X)$, 
the two norms on $\mathscr{B}(X)/I$ are equivalent. 
We conclude by setting $I= \mathscr{K}(X)$ or $I= \mathscr{W}(X)$.
\end{proof}
%\begin{proof} Combine \cite[Theorem 3.3]{johnson} with \cite[Proposition 2.1.5]{da}.   
%Suppose that $X$ is infinite-dimensional and satisfies $X\cong X\oplus X$ (linear homeomorphism). Then every algebra homomorphism from $\mathscr{B}(X)$ into a Banach algebra is continuous by \cite[Theorem 3.3]{johnson} (note that Johnson's result applies to a slightly larger class of Banach spaces, namely those with a continued bisection; we could have stated the proposition in those terms). Let $I$ be a closed ideal of $\mathscr{B}(X)$ and consider a complete algebra norm $|||\cdot|||$ on $\mathscr{B}(X)/I$ different from the given norm $||\cdot||$. Denote the quotient map by $Q: \mathscr{B}(X)\to (\mathscr{B}(X)/I,||\cdot||)$, and the identity map by $\iota: ( \mathscr{B}(X)/I, ||\cdot||)\to ( \mathscr{B}(X)/I,|||\cdot|||)$. Then $\iota \circ Q:  \mathscr{B}(X) \to ( \mathscr{B}(X)/I, |||\cdot|||)$ is an algebra homomorphism, so is continuous. Being a quotient map, $Q$ is an open mapping; therefore $\iota$ is continuous, and hence there exists $C>0$ such that $|||b+I||| \leq C||b+I||$ for every $b\in \mathscr{B}(X)$. Next, the Banach Isomorphism Theorem yields $c>0$ such that $||b+I|| \leq c|||b+I|||$ for every $b\in \mathscr{B}(X)$, and so the two norms on $\mathscr{B}(X)/I$ are equivalent.
%We conclude by setting $I= \mathscr{K}(X)$ or $I= \mathscr{W}(X)$.
%\end{proof}

The following Banach spaces have a Calkin algebra with a unique complete norm. The purpose of the list is not to give a comprehensive account, but rather a flavour of the wide variety of spaces sharing the property. Examples (ii)-(viii) follow because $X\cong X\oplus X$.
%, as was mostly known to Banach \cite{Ban}.
\begin{romanenumerate}
\item Any finite-dimensional Banach space.
\item $\ell_p$ for $1\leq p\leq \infty$ and $c_0$. This generalises to $\ell_p(\mathbb{I})$ for $1\leq p\leq \infty$ and $c_0(\mathbb{I})$ for any index set $\mathbb{I}$.
\item $\ell_p\oplus \ell_q$ for $1\leq p<q\leq \infty$, and $c_0\oplus \ell_q$ for $1\leq q\leq \infty$. We may generalise this to $X\oplus Y$ where $X,Y$ are Banach spaces such that $X\cong X\oplus X$ and $Y\cong Y\oplus Y$. 
\item $(\bigoplus_{n=1}^{\infty} \ell_r^n)_{\ell_p}$ and $(\bigoplus_{n=1}^{\infty} \ell_r^n)_{c_0}$ for $p,r\in [1,\infty]$.  
%In fact this holds for finite direct sums of the form $\ell_{p_1}\oplus \cdots \oplus \ell_{p_n}$ (resp. $c_0\oplus \ell_{p_1}\oplus \cdots \oplus \ell_{p_n}$) for $n\in \N$ and $1\leq p_j\leq \infty$, and arbitrary index sets.
\item $C(K)$ for $K$ an infinite compact metric space.
%\item  $C[0,\kappa]$ for any countable ordinal $\kappa$.
\item $L_p[0,1]$ for $1\leq p\leq \infty$.
\item $C^{(n)}[0,1]$, the $n$-times continuously differentiable functions on $[0,1]$, for all $n\in \N$.
\item Tsirelson's space $T$, and its dual $T^*$.
\item $J_p$ for $1<p<\infty$, the $p$th James space. For $p=2$ this follows from a result of Loy and Willis \cite[Theorem 2.7]{LW}, combined with \cite[Remark 3.9]{ST} of Saksman and Tylli, while Laustsen \cite[Proposition 4.9]{Laus} showed later that $\mathscr{E}(J_p)= \mathscr{K}(J_p)$ for general $1<p<\infty$.
\item $C[0,\omega_1]$, where $\omega_1$ denotes the first uncountable ordinal. Ogden \cite{Og} has shown that all homomorphisms from $\mathscr{B}(C[0,\omega_1])$ into a Banach algebra are continuous, so we deduce the result by the same reasoning as in the proof of Proposition \ref{isomorphictosquare}.
\item $X_{\normalfont{\text{AH}}}$, the Argyros-Haydon solution to the scalar-plus-compact problem \cite{AH}. In this remarkable case we have $\mathscr{B}(X_{\normalfont{\text{AH}}})/\mathscr{K}(X_{\normalfont{\text{AH}}})\cong \mathbb{K}$.
\item $X_{k}$ for $1\leq k\leq \infty$, Tarbard's reworkings of the Argyros-Haydon construction \cite{Tar1}, \cite{Tar2}. For $1\leq k< \infty$ these spaces have finite-dimensional Calkin algebras. The space $X_{\infty}$ has a Calkin algebra isometrically isomorphic to the Banach algebra $\ell_1(\N_0)$ with convolution product, which is semisimple \cite[Theorem 4.6.9(i)]{da}.
\item For every countable compact metric space $M$, Motakis, Puglisi and Zisimopoulou \cite{MPZ} have constructed a Banach space $X_M$ with Calkin algebra isomorphic (as a Banach algebra) to $C(M)$. These Calkin algebras are semisimple \cite[Corollary 3.2.13]{da} and thus have a unique complete norm.
\item $X_{\normalfont{\text{KL}}}$, Kania and Laustsen's \cite{LK} variant of the Argyros-Haydon space whose Calkin algebra is 3-dimensional.
\end{romanenumerate}

\begin{remark} \label{uniquenorm} One can also ask whether a Banach algebra has a unique norm, dropping the completeness assumption. This stronger property holds for the Calkin algebra on certain Banach spaces. Indeed, Meyer \cite{May} has shown that the Calkin algebras of $\ell_p$ for $1\leq p< \infty$ and $c_0$ have a unique norm, and his results were substantially extended by Ware \cite{Ware} to cases (iii) (excluding $q=\infty$, and the generalisation), (iv) (excluding $p=\infty$), (viii), and (ix), listed above. Ware's well-written and carefully referenced thesis \cite{Ware} is an excellent source to find out more about the topic; his perspective differs from ours in that he hardly considers uniqueness of \emph{complete} norms, instead focusing on the more stringent unique-norm condition. \\
Zem\'anek \cite{Zem} noticed that a result of Astala and Tylli provided the first example of a Calkin algebra with a non-unique norm;
later Tylli \cite[Remark 1]{Ty3} gave a related result showing the same is possible for the weak Calkin algebra. 
Ware \cite[Example 3.2.8]{Ware} demonstrated two more examples of Calkin algebras with a non-unique norm: one is the Tarbard space $X_\infty$ in (xii), and the other is a Banach space of Read \cite{read} that we will introduce more fully in a moment. However, all these norms that are not equivalent to the usual quotient norm fail to be complete, so shed no light on our problem.
\end{remark}

The following Banach spaces have a weak Calkin algebra with a unique complete norm. \\
\begin{romanenumerate}
\item Any Banach space $X$ such that $X\cong X\oplus X$, and any finite-dimensional Banach space; in particular, (i)-(viii) from the previous list.
\item Any reflexive Banach space, since $\mathscr{B}(X)/\mathscr{W}(X)= \{0\}$.
\item  $J_p$ for $1<p<\infty$, due to the fact that $\mathscr{B}(J_p)/\mathscr{W}(J_p)\cong \mathbb{K}$ \cite[p. 225]{EM}.
\item $J_p\oplus J_q$ for $1<p\leq q< \infty$. This follows by considering the operators on $J_p\oplus J_q$ as operator-valued $2\times 2$ matrices, and noting Loy and Willis' result \cite[Theorem 4.5]{LW} that for $1<p < q< \infty$, $\mathscr{B}(J_q,J_p)=\mathscr{K}(J_q,J_p)$ and $\mathscr{B}(J_p,J_q)= \mathscr{W}(J_p,J_q)\oplus \mathbb{K} I_{p,q}$, where $I_{p,q}$ denotes the formal inclusion. More generally, this holds for finite direct sums of James spaces.
\item $C[0,\omega_1]$, $X_{\normalfont{\text{AH}}}$, $X_k$ for $1\leq k \leq \infty$, and $X_{\normalfont{\text{KL}}}$. These Banach spaces are each preduals of $\ell_1(\mathbb{I})$ for some index set $\mathbb{I}$, and so, by the theorems of Schauder and Gantmacher (see, \emph{e.g.} \cite[Theorem 3.4.15 and Theorem 3.5.13]{Meg}), the weakly compact operators coincide with the compacts.
\end{romanenumerate}

\section{Calkin algebras with non-equivalent complete norms}
 
Our first lemma is well-known, but we give a proof for completeness. It shows that one condition for a Banach algebra to lack a unique complete norm is to lack any sort of algebra structure at all. Before we make a precise statement, recall that the \emph{unitisation} $\widetilde{A}$ of a Banach algebra $A$ (with scalar field $\mathbb{K}$) is the vector space $A\oplus \mathbb{K}$ equipped with the norm $||(a,\alpha)||= ||a||_A+|\alpha|$ and the product $(a,\alpha)(b,\beta)= (ab+\alpha b+\beta a, \alpha\beta)$. This turns $\widetilde{A}$ into a unital Banach algebra containing $A$ isometrically as an ideal of codimension one.

\begin{lemma} \label{nonunique} Let $(A,||\cdot||_A)$ be an infinite-dimensional Banach algebra such that $ab=0$ for every $a,b\in A$. Then $\widetilde{A}$ does not have a unique complete norm.
\end{lemma}
\begin{proof} Since $A$ is infinite-dimensional, there is a discontinuous linear bijection $\gamma: A\to A$. To see this, take a Hamel basis $(b_j)_{j\in J}$ for $A$ ($J$ some uncountable index set). 
%Without loss of generality we can take $\|b_j\|=1$ for each $j\in J$. 
Since $J$ is uncountable, we may write it as the disjoint union of an infinite sequence $(j_n)_{n\in \N}$ and an uncountable subset $J_0$. Define $\gamma(b_{j_n})= nb_{j_n}$ for $n\in \N$ and $\gamma({b_j})=b_j$ for $j\in J_0$, and extend by linearity. Then $\gamma$ is a bijection because it maps a basis onto a basis, and it is unbounded.
Define a new norm on $\widetilde{A}$ by
\begin{equation*} |||(a,\alpha)|||= ||\gamma(a)||_A+ |\alpha|
\end{equation*} for each $(a,\alpha)$ in $\widetilde{A}$. This is easily checked to be an algebra norm, since $\gamma$ is a linear bijection. Let $((a_n,\alpha_n))_{n=1}^{\infty}$ be Cauchy in $(\widetilde{A},|||\cdot|||)$. Then for each $\epsilon>0$ there exists $n_0$ such that 
\begin{equation*} |||(a_n,\alpha_n)-(a_m,\alpha_m)|||= ||\gamma(a_n-a_m)||_A+|\alpha_n-\alpha_m| <\epsilon
\end{equation*} for all $n,m\geq n_0$. So $(\alpha_n)_{n=1}^{\infty}$ is Cauchy in $\mathbb{K}$ and $(\gamma(a_n))_{n=1}^{\infty}$ is Cauchy in $(A,||\cdot||_A)$. Since $(A,||\cdot||_A)$ and $\mathbb{K}$ are complete there exist $\alpha \in \mathbb{K}$ and $a\in A$ which are the respective limits. 
%Note that $b\in B$ is the unique element of $B$ such that $b+I=c+I$ for some $c\in A$. 
Then because $\gamma$ is a bijection we have
\begin{equation*} |||(a_n,\alpha_n)- (\gamma^{-1}(a),\alpha)|||= ||\gamma(a_n-\gamma^{-1}(a))||_A+|\alpha_n-\alpha| \to 0
\end{equation*} as $n\to \infty$. Since $(\gamma^{-1}(a),\alpha)$ is in $\widetilde{A}$ we conclude that $\widetilde{A}$ is complete with respect to $|||\cdot|||$. 
Now suppose that there exists $C>0$ such that $|||(a,\alpha)|||\leq C||(a,\alpha)||$ for every $(a,\alpha)\in \widetilde{A}$. Then $|||(a,0)|||\leq C||(a,0)||$ for every $a\in A$, and so $||\gamma(a)||_A\leq C||a||_A$ for each $a\in A$; but $\gamma$ is unbounded, a contradiction. Therefore $|||\cdot|||$ is not equivalent to $||\cdot||$ on $\widetilde{A}$.
\end{proof}

We shall refer to a Banach algebra $A$ such that $ab=0$ for every $a,b \in A$ as having the \emph{trivial product}.
This lemma may seem a small observation, but it can be applied to certain quotient algebras quite effectively. Firstly, let us see that there is a Banach space with a Calkin algebra lacking a unique complete norm. Before we state the result we require one more piece of terminology. A bounded operator on a Banach space $X$ is \emph{strictly singular} if it is not bounded below on any infinite-dimensional subspace. We write $\mathscr{S}(X)$ for the closed ideal of $\mathscr{B}(X)$ consisting of strictly singular operators, and record the standard fact that $\mathscr{K}(X)\subseteq \mathscr{S}(X)$.
Argyros and Motakis \cite[Theorem A]{AM2} have recently constructed a Banach space with the following remarkable properties (the result for complex scalars is noted in the comments after \cite[Theorem B]{AM2}). 

\begin{theorem}[Argyros, Motakis] \label{AM} There exists a Banach space $X_{\normalfont{\text{AM}}}$ such that 
\begin{romanenumerate}
\item $X_{\normalfont{\text{AM}}}$ is reflexive and has a Schauder basis;
\item every operator $T\in \mathscr{B}(X_{\normalfont{\text{AM}}})$ can be uniquely written as $T=\lambda I_{X_{\normalfont{\text{AM}}}}+S$ for some $\lambda \in \mathbb{K}$ and $S$ strictly singular ($I_{X_{\normalfont{\text{AM}}}}$ denotes the identity operator);
\item the composition of any two strictly singular operators on $X_{\normalfont{\text{AM}}}$ is compact;
\item $\mathscr{S}(X_{\normalfont{\text{AM}}})$ is non-separable.
\end{romanenumerate}
\end{theorem}

\begin{theorem} \label{AMCalkin} The Calkin algebra $\mathscr{B}(X_{\normalfont{\text{AM}}})/\mathscr{K}(X_{\normalfont{\text{AM}}})$ of the space of Argyros and Motakis has at least two non-equivalent complete algebra norms. 
\end{theorem}
\begin{proof} Observe that by Theorem \ref{AM}(ii) we have 
\begin{equation*} \mathscr{B}(X_{\normalfont{\text{AM}}})/\mathscr{K}(X_{\normalfont{\text{AM}}})=\mathscr{S}(X_{\normalfont{\text{AM}}})/\mathscr{K}(X_{\normalfont{\text{AM}}}) \oplus \mathbb{K} (I_{X_{\normalfont{\text{AM}}}}+\mathscr{K}(X)) \end{equation*}
and so the Calkin algebra is the unitisation of the Banach algebra 
$\mathscr{S}(X_{\normalfont{\text{AM}}})/\mathscr{K}(X_{\normalfont{\text{AM}}})$.
We shall show that $\mathscr{S}(X_{\normalfont{\text{AM}}})/\mathscr{K}(X_{\normalfont{\text{AM}}})$ is infinite-dimensional and has the trivial product, and then implement Lemma \ref{nonunique}.
%giving at least two non-equivalent complete algebra norms on $\mathscr{B}(X_{\normalfont{\text{AM}}})/\mathscr{K}(X_{\normalfont{\text{AM}}})$ by Lemma \ref{nonunique}.
%thus satisfying the conditions of Lemma \ref{nonunique} and giving at least two non-equivalent complete algebra norms on $\mathscr{B}(X_{\normalfont{\text{AM}}})/\mathscr{K}(X_{\normalfont{\text{AM}}})$.
Suppose towards a contradiction that $\mathscr{S}(X_{\normalfont{\text{AM}}})/\mathscr{K}(X_{\normalfont{\text{AM}}})$ is finite-dimensional; then it is separable.
By Theorem \ref{AM}(i), $X_{\normalfont{\text{AM}}}$ is separable (since it has a basis) and reflexive, so $X_{\normalfont{\text{AM}}}^*$ is separable. Take countable dense subsets $M$ and $N$ of $X_{\normalfont{\text{AM}}}$ and $X_{\normalfont{\text{AM}}}^*$, respectively. Then it is easy to check that $\spa\{x\otimes f: x\in M, f\in N\}$ is dense in $\mathscr{F}(X_{\normalfont{\text{AM}}})$ (the ideal of finite rank operators). Since $X_{\normalfont{\text{AM}}}$ has a basis, $\mathscr{F}(X_{\normalfont{\text{AM}}})$ is dense in $\mathscr{K}(X_{\normalfont{\text{AM}}})$, and so in fact $\spa\{x\otimes f: x\in M, f\in N\}$ is dense in $\mathscr{K}(X_{\normalfont{\text{AM}}})$. Therefore $\mathscr{K}(X_{\normalfont{\text{AM}}})$ is separable.
But the fact that separability is a three-space property \cite[Corollary 1.12.10]{Meg} means that $\mathscr{S}(X_{\normalfont{\text{AM}}})$ must also be separable, contradicting Theorem \ref{AM}(iv). Therefore $\mathscr{S}(X_{\normalfont{\text{AM}}})/\mathscr{K}(X_{\normalfont{\text{AM}}})$ is infinite-dimensional. The fact that it has the trivial product follows from Theorem \ref{AM}(iii), and so we may apply Lemma \ref{nonunique} to obtain the result.
\end{proof}

\begin{corollary} The Calkin algebra $\mathscr{B}(X_{\normalfont{\text{AM}}}^*)/\mathscr{K}(X_{\normalfont{\text{AM}}}^*)$ has at least two non-equivalent complete algebra norms. 
\end{corollary}
\begin{proof} (cf. \cite[Theorem 1.2.2]{Ware}) By Theorem \ref{AM}(i), $X_{\normalfont{\text{AM}}}$ is reflexive, and so the map
 \begin{equation*} T+\mathscr{K}(X_{\normalfont{\text{AM}}})\mapsto T^*+\mathscr{K}(X_{\normalfont{\text{AM}}}^*),\quad \mathscr{B}(X_{\normalfont{\text{AM}}})/\mathscr{K}(X_{\normalfont{\text{AM}}})\to \mathscr{B}(X_{\normalfont{\text{AM}}}^*)/\mathscr{K}(X_{\normalfont{\text{AM}}}^*) \end{equation*}
  is an isometric anti-isomorphism by Schauder's Theorem. 
%$T$ is compact if and only if $T^*$ is compact, and so $T+\mathscr{K}(X_{\normalfont{\text{AM}}}) \mapsto T^*+\mathscr{K}(X_{\normalfont{\text{AM}}}^*)$ is also a bijective isometric antihomomorphism from $\mathscr{B}(X_{\normalfont{\text{AM}}})/\mathscr{K}(X_{\normalfont{\text{AM}}})$ to $\mathscr{B}(X_{\normalfont{\text{AM}}}^*)/\mathscr{K}(X_{\normalfont{\text{AM}}}^*)$.
Theorem \ref{AMCalkin} allows us to choose two non-equivalent complete algebra norms on $\mathscr{B}(X_{\normalfont{\text{AM}}})/\mathscr{K}(X_{\normalfont{\text{AM}}})$, and these pass to non-equivalent complete algebra norms on $\mathscr{B}(X_{\normalfont{\text{AM}}}^*)/\mathscr{K}(X_{\normalfont{\text{AM}}}^*)$ using the anti-isomorphism.
%since Schauder's Theorem guarantees that $\mathscr{B}(X_{\normalfont{\text{AM}}})/\mathscr{K}(X_{\normalfont{\text{AM}}})$ and $\mathscr{B}(X_{\normalfont{\text{AM}}}^*)/\mathscr{K}(X_{\normalfont{\text{AM}}}^*)$ are also isometrically anti-isomorphic.
%Define norms on $\mathscr{B}(X_{\normalfont{\text{AM}}}^*)/\mathscr{K}(X_{\normalfont{\text{AM}}}^*)$ by $||T^*+\mathscr{K}(X_{\normalfont{\text{AM}}}^*)||_3=||T+\mathscr{K}(X_{\normalfont{\text{AM}}})||_1$ and $||T^*+\mathscr{K}(X_{\normalfont{\text{AM}}}^*)||_4=||T+\mathscr{K}(X_{\normalfont{\text{AM}}})||_2$. Then using the bijective antihomomorphism it is easy to check that $||\cdot||_3$ and $||\cdot||_4$ are non-equivalent complete algebra norms on the Calkin algebra of $X_{\normalfont{\text{AM}}}^*$.
\end{proof}

%\begin{remark} Since $X_{\normalfont{\text{AM}}}$ is reflexive we see that the weak Calkin algebra of $X_{\normalfont{\text{AM}}}$ has a unique complete norm. The same applies to $X_{\normalfont{\text{AM}}}^*$.
%\end{remark}
Since $X_{\normalfont{\text{AM}}}$ is reflexive we see that the weak Calkin algebra of $X_{\normalfont{\text{AM}}}$ has a unique complete norm. 
What can we say about weak Calkin algebras in general? We have seen that for many Banach spaces $X$, $\mathscr{B}(X)/\mathscr{W}(X)$ has a unique complete norm. Let us now give an example where this is not true. In \cite{read}, Read constructed a Banach space $E_{\normalfont{\text{R}}}$ which has a discontinuous derivation from $\mathscr{B}(E_{\normalfont{\text{R}}})$ into a Banach $\mathscr{B}(E_{\normalfont{\text{R}}})$-bimodule, the first example of a Banach space with such a property.
Our next theorem follows directly from his results, using Lemma \ref{nonunique} (Read states his results for complex scalars only, but they carry over verbatim to the real case).
%Laustsen and the present author made a thorough study of $E_{\normalfont{\text{R}}}$ in the context of splittings of extensions of Banach algebras \cite{LS}, and we use a key theorem from \cite{LS} to prove our next result. 
%We quickly recall a definition used frequently in \cite{LS}.
%Let $A$ be a Banach algebra. An \emph{extension} of $A$ is a short exact sequence of Banach algebras and continuous algebra homomorphisms 
%\begin{equation*}
%\{0\} \longrightarrow I \overset{\iota}\longrightarrow B \overset{\pi}{\longrightarrow} A \longrightarrow \{0\}.
%\end{equation*}  
%The extension \emph{splits strongly} if there is a continuous algebra homomorphism $\theta: A\rightarrow B$ such that $\pi\circ \theta=\mathrm{id}_A$. 
%In \cite[Theorem 1.3]{LS} the following is proved about Read's space.

%\begin{theorem} \label{WEBEsplitexact} There exists a continuous, 
%surjective algebra homomorphism~$\pi$ 
%from~$\mathscr{B}(E_{\normalfont{\text{R}}})$
%onto~$\ell_2(\N)^{\sim}$, when $\ell_2(\N)$ is given the trivial product, such that the extension
%\begin{equation}\label{WEBEsplitexactEq} 
 % \spreaddiagramcolumns{2ex}\xymatrix{\{0\}\ar[r] &
  %  \mathscr{W}(E_{\normalfont{\text{R}}})\ar[r] &
  %  \mathscr{B}(E_{\normalfont{\text{R}}})\ar^-{\displaystyle{\pi}}[r]
   % & \ell_2(\N)^{\sim}\ar[r] & \{0\}}
%\end{equation} splits strongly.
%\end{theorem}

\begin{theorem} The weak Calkin algebra $\mathscr{B}(E_{\normalfont{\text{R}}})/\mathscr{W}(E_{\normalfont{\text{R}}})$ of the space of Read has at least two non-equivalent complete algebra norms. 
\end{theorem}
\begin{proof} In \cite[see p. 306]{read}, Read proves that there is an ideal $J$ of codimension one in $\mathscr{B}(E_{\normalfont{\text{R}}})$ such that $J^2\subseteq \mathscr{W}(E_{\normalfont{\text{R}}})\subseteq J$ (where $J^2= \spa\{ab: a,b\in J\}$), and $\mathscr{B}(E_{\normalfont{\text{R}}})/\mathscr{W}(E_{\normalfont{\text{R}}})$ is infinite-dimensional. From this, it follows that 
\begin{equation*} \mathscr{B}(E_{\normalfont{\text{R}}})/\mathscr{W}(E_{\normalfont{\text{R}}})= J/\mathscr{W}(E_{\normalfont{\text{R}}}) \oplus \mathbb{K}(I_{E_{\normalfont{\text{R}}}}+\mathscr{W}(E_{\normalfont{\text{R}}}))
\end{equation*}
where $I_{E_{\normalfont{\text{R}}}}$ is the identity operator. Since $J^2\subseteq \mathscr{W}(E_{\normalfont{\text{R}}})$, we see immediately that $J/\mathscr{W}(E_{\normalfont{\text{R}}})$ has the trivial product, and the fact that $\mathscr{B}(E_{\normalfont{\text{R}}})/\mathscr{W}(E_{\normalfont{\text{R}}})$ is infinite-dimensional ensures that $J/\mathscr{W}(E_{\normalfont{\text{R}}})$ is too. Thus we may apply Lemma \ref{nonunique} to the Banach algebra $J/\mathscr{W}(E_{\normalfont{\text{R}}})$ to obtain the result.
\end{proof}

%\begin{proof} Consider the Hilbert space $\ell_2(\N)$. Equip $\ell_2(\N)$ with the trivial product to make it a Banach algebra, and then take the unitisation $\ell_2(\N)^{\sim}$. In \cite[Theorem 1.3]{LS} it is shown that there is a continuous surjective algebra homomorphism $\psi: \mathscr{B}(E_{\normalfont{\text{R}}}) \to \ell_2(\N)^{\sim}$ with $\ker \psi= \mathscr{W}(E_{\normalfont{\text{R}}})$. By the Fundamental Isomorphism Theorem it follows that $\mathscr{B}(E_{\normalfont{\text{R}}})/\mathscr{W}(E_{\normalfont{\text{R}}}) \cong \ell_2(\N)^{\sim}$ as Banach algebras. Now Lemma \ref{nonunique} implies that $\ell_2(\N)^{\sim}$ can be endowed with non-equivalent complete algebra norms. By moving these norms onto $\mathscr{B}(E_{\normalfont{\text{R}}})/\mathscr{W}(E_{\normalfont{\text{R}}})$ via the isomorphism we obtain the result.
%\end{proof}

\begin{remark} We do not know if $\mathscr{B}(E_{\normalfont{\text{R}}})/\mathscr{K}(E_{\normalfont{\text{R}}})$ has a unique complete norm.
\end{remark}

%Both the counterexamples we have used are `pathological' Banach spaces in some sense, and for most classical spaces we get uniqueness. This leads to the following problem.

%\begin{quote}
%\textsl{Characterise the Banach spaces such that the Calkin algebra/weak Calkin algebra has a unique complete norm.}
%\end{quote}

%However, considering the spaces listed in Section \ref{section1}, which vary from classical sequence spaces to modern HI constructions, the problem may not have an easy answer. \\

\noindent \textbf{Acknowledgements.}
The author would like to thank Dr N.J. Laustsen for his help and guidance during the preparation of this note, and acknowledges financial support from the Lancaster University Faculty of Science and Technology. 

\bibliographystyle{amsplain}

\end{document}